\documentclass[a4paper,10pt]{amsart}
\usepackage{geometry,float}
\usepackage{enumitem}
\usepackage{hyperref}
\usepackage{xcolor}
\usepackage{mathtools} 

\usepackage{tikz}
\usetikzlibrary{matrix,arrows}

\newcommand{\NN}{\mathbb{N}}
\newcommand{\RR}{\mathbb{R}}
\newcommand{\CC}{\mathbb{C}}
\newcommand{\HH}{\mathbb{H}}

\newcommand{\M}[2]{M({#1},{#2})}

\newcommand{\kk}{\mathfrak{k}}

\DeclareMathOperator{\im}{im}

\newcommand{\imp}{\mathfrak{m}}

\newtheorem{proposition}{Proposition}[section]

\newtheorem{lemma}[proposition]{Lemma}
\newtheorem{corollary}[proposition]{Corollary}

\newtheorem{consequence}[proposition]{Consequence}

\newtheorem{innercustomthm}{Theorem}
\newenvironment{customthm}[1]
  {\renewcommand\theinnercustomthm{#1}\innercustomthm}
  {\endinnercustomthm}

\theoremstyle{remark}

\newtheorem{remark}[proposition]{Remark}

\theoremstyle{definition}
\newtheorem{definition}[proposition]{Definition}

\newcommand{\En}[2]{
  \begin{tikzpicture}[scale=.5]
    \draw (-1,1) node[anchor=east]  {#2};
    \draw[thick,xshift=0 cm] (0 cm,0) circle (3 mm) node[below=1mm] {$1$};
    \draw[thick] (4 cm,2 cm) circle (3 mm) node[above=1mm] {$2$};
    \ifx3#1
    \else
	    \draw[thick] (4 cm, 3mm) -- +(0, 1.4 cm);
	\fi
    \foreach \x in {3,...,#1}
      \draw[thick,xshift=\x cm - 2 cm] (\x cm - 2 cm,0) circle (3 mm) node[below=1mm] {$\x$};
    \foreach \y in {3,...,#1}
      \draw[thick,xshift=\y cm - 3 cm] (\y cm - 3 cm,0) ++(.3 cm, 0) -- +(14 mm,0);
  \end{tikzpicture}
}

\begin{document}

\title{Generalized spin representations. \\ Part 2: Cartan--Bott periodicity for the split real $E_n$ series}
\author{Max Horn}
\address{JLU Giessen, Mathematisches Institut, Arndtstrasse 2, 35392 Giessen, Germany}
\email{max.horn@math.uni-giessen.de }
\author{Ralf K\"ohl (n\'e Gramlich)}
\address{JLU Giessen, Mathematisches Institut, Arndtstrasse 2, 35392 Giessen, Germany}
\email{ralf.koehl@math.uni-giessen.de }

\maketitle
\begin{abstract}
In this article we analyze the quotients of the maximal compact subalgebras of the split real Kac--Moody algebras of the $E_n$ series resulting from the generalized spin representation introduced in \cite{Hainke/Koehl/Levy}. It turns out that these quotients satisfy a Cartan--Bott periodicity.

Our findings are also meaningful in the finite-dimensional cases of $A_2 \oplus A_1$, $A_4$, $D_5$, $E_6$, $E_7$, $E_8$, where it turns out that the generalized spin representation is injective. Consequently the observed Cartan--Bott periodicity provides a structural explanation for the seemingly sporadic isomorphism types of the maximal compact Lie subalgebras of the split real Lie algebras of types $E_6$, $E_7$, $E_8$.
\end{abstract}

\section{Introduction}

In this article we continue the investigation of the generalized spin representations introduced in the first part \cite{Hainke/Koehl/Levy}. We focus on the $E_n$ series and use the original description of the generalized spin representation from \cite{DamourKleinschmidtNicolai}, \cite{deBuylHenneauxPaulot}, \cite{Hainke/Koehl/Levy} via Clifford algebras.

The $E_n$ series is traditionally only defined for $n\in\{6,7,8\}$. However,
using the Bourbaki style labeling shown in Figure~\ref{fig:En}, it naturally
extends to arbitrary $n\in\NN$. Using this description, one has
$E_1=A_1$, $E_2=A_1\oplus A_1$, $E_3=A_2\oplus A_1$, $E_4=A_4$, $E_5=D_5$
(see Figure~\ref{fig:E3-to-E8}).

\begin{figure}[h]
  \begin{tikzpicture}[scale=.5]
    \draw (-1,1) node[anchor=east]  {$E_n$};
    \draw[thick,xshift=0 cm] (0 cm,0) circle (3 mm) node[below=1mm] {$1$};
    \foreach \x in {3,...,6}
      \draw[thick,xshift=\x cm - 2 cm] (\x cm - 2 cm,0) circle (3 mm) node[below=1mm] {$\x$};
    \foreach \y in {0,...,3}
      \draw[thick,xshift=\y cm] (\y cm,0) ++(.3 cm, 0) -- +(14 mm,0);
    \draw[thick,xshift=5 cm] (5 cm,0) circle (3 mm) node[below=1mm] {$n$};
    \draw[thick,xshift=4 cm,dashed] (4 cm,0) ++(.3 cm, 0) -- +(14 mm,0);
    \draw[thick] (4 cm,2 cm) circle (3 mm) node[above=1mm] {$2$};
    \draw[thick] (4 cm, 3mm) -- +(0, 1.4 cm);
  \end{tikzpicture}
\caption{The Dynkin diagram of type $E_n$}
\label{fig:En}
\end{figure}
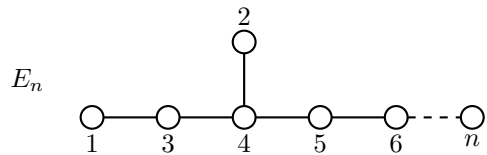

An elementary combinatorial counting argument using binomial coefficients allows us to determine lower bounds for the $\RR$-dimension of the images of the generalized spin representation. These images have to be compact, whence reductive by \cite[Theorem~4.11]{Hainke/Koehl/Levy} and even semisimple, if the diagram be irreducible, thus providing an upper bound for the $\RR$-dimension via the maximal compact Lie subalgebras of the Clifford algebras. As it turns out, the lower and the upper bounds coincide, providing the following Cartan--Bott periodicity.

\begin{customthm}{A}[Cartan--Bott periodicity of the $E_{n}$ series]\label{thm:A}
Let $n \in \NN$ with $n\geq 4$, let $\mathfrak{k}$ be the maximal compact Lie subalgebra of the split real Kac--Moody Lie algebra of type $E_n$, let $C=C(\RR^n,q)$ be the Clifford algebra with respect to the standard positive definite quadratic form $q$ and let $\rho : \mathfrak{k} \to C$ be the standard generalized spin representation.

Then $\im(\rho)$ is isomorphic to
\begin{enumerate}
\setcounter{enumi}{-1}
\item $\mathfrak{so}(2^{\frac{n}{2}})) \leq \RR\otimes_\RR \M{2^{\frac{n}{2}}}{\RR}$, if $n \equiv 0 \pmod 8$,
\item $\mathfrak{so}(2^{\frac{n-1}{2}}) \oplus \mathfrak{so}(2^{\frac{n-1}{2}})  \leq \left( \RR \oplus \RR \right) \otimes_\RR \M{2^{\frac{n-1}{2}}}{\RR}$, if $n \equiv 1 \pmod 8$,
\item $\mathfrak{so}(2^{\frac{n}{2}}) \leq \M{2}{\RR} \otimes_\RR \M{2^{\frac{n-2}{2}}}{\RR}$, if $n \equiv 2 \pmod 8$,
\item $\mathfrak{su}(2^{\frac{n-1}{2}}) \leq \M{2}{\CC} \otimes_\RR \M{2^{\frac{n-3}{2}}}{\RR}$, if $n \equiv 3 \pmod 8$,
\item $\mathfrak{sp}(2^{\frac{n-2}{2}}) \leq \M{2}{\HH} \otimes_\RR \M{2^{\frac{n-4}{2}}}{\RR}$, if $n \equiv 4 \pmod 8$,
\item $\mathfrak{sp}(2^{\frac{n-3}{2}})\oplus \mathfrak{sp}(2^{\frac{n-3}{2}}) \leq \left( \M{2}{\HH} \oplus \M{2}{\HH} \right) \otimes_\RR \M{2^{\frac{n-5}{2}}}{\RR}$, if $n \equiv 5 \pmod 8$,
\item $\mathfrak{sp}(2^{\frac{n-2}{2}}) \leq \M{4}{\HH} \otimes_\RR \M{2^{\frac{n-6}{2}}}{\RR}$, if $n \equiv 6 \pmod 8$,
\item $\mathfrak{su}(2^{\frac{n-1}{2}}) \leq \M{8}{\CC} \otimes_\RR \M{2^{\frac{n-7}{2}}}{\RR}$, if $n \equiv 7 \pmod 8$,
\end{enumerate}
i.e., $\im(\rho)$ is a semisimple maximal compact Lie subalgebra of $C$.
\end{customthm}

\medskip
Along the way we arrive at a structural explanation for the isomorphism types of the maximal compact Lie subalgebras of the semisimple split real Lie algebras of types 
$E_3=A_2\oplus A_1$, $E_4=A_4$, $E_5=D_5$, $E_6$, $E_7$, $E_8$.

\begin{customthm}{B} \label{thm:B}
The maximal compact Lie subalgebras of the semisimple split real Lie algebras of types
$A_2\oplus A_1$, $A_4$, $D_5$, $E_6$, $E_7$, $E_8$ are isomorphic to
$\mathfrak{u}(2)$, 
$\mathfrak{sp}(2)\cong\mathfrak{so}(5)$, 
$\mathfrak{sp}(2)\oplus\mathfrak{sp}(2)\cong\mathfrak{so}(5)\oplus\mathfrak{so}(5)$, 
$\mathfrak{sp}(4)$, 
$\mathfrak{su}(8)$, 
$\mathfrak{so}(16)$, 
respectively.
\end{customthm}

\bigskip \noindent
\textbf{Acknowledgements.}
We thank Klaus Metsch for pointing out to us the identity of sums of binomial coefficients in 
Proposition~\ref{binomial}.
  This research has been partially funded by the EPRSC grant 
EP/H02283X. The second author gratefully acknowledges the hospitality of the IHES at Bures-sur-Yvette and of the Albert Einstein Institute at Golm.

\begin{figure}[h]
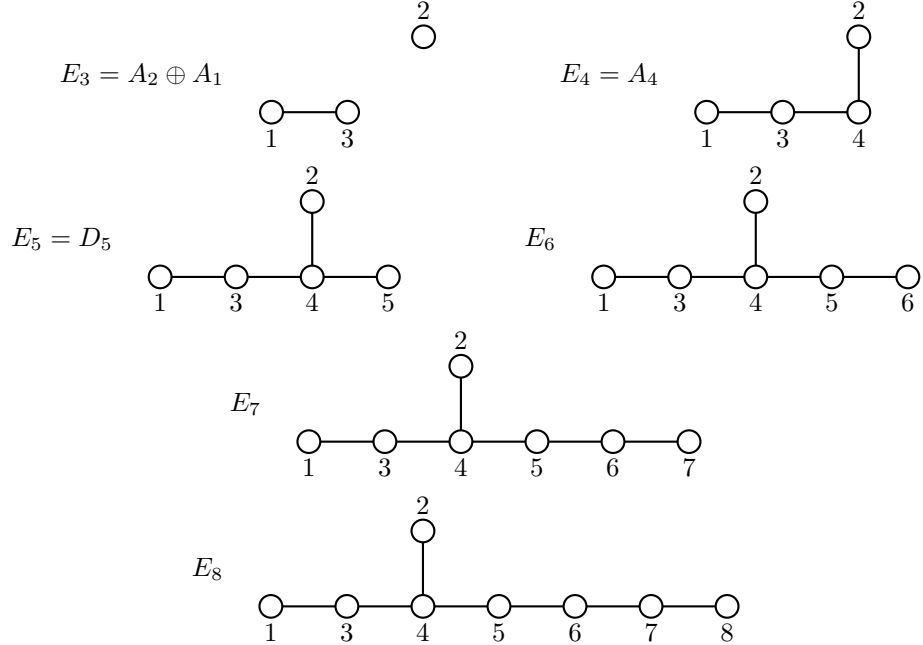

\En{3}{$E_3=A_2\oplus A_1$}
  \hspace{1cm}
\En{4}{$E_4=A_4$}

\En{5}{$E_5=D_5$}
  \hspace{1cm}
\En{6}{$E_6$}

\En{7}{$E_7$}

\En{8}{$E_8$}
\caption{The Dynkin diagrams of types $E_3$ to $E_8$.}
\label{fig:E3-to-E8}
\end{figure}

\section{Cartan--Bott periodicity of Clifford algebras} \label{sec:cartan-bott}

Let $\NN = \{1,2,3,\ldots\}$ be the set of natural numbers, and let $\RR$, $\CC$, resp.\ $\HH$ denote the reals, complex numbers resp.\ quaternions. For $n\in\NN$ and a division ring $\mathbb{D}$, denote by $M(n,\mathbb{D})$ the $\mathbb{D}$-algebra of $n\times n$ matrices over $\mathbb{D}$.

Let $V$ be an $\RR$-vector space and $q\colon V \to \RR$ a quadratic form with associated bilinear form $b$. Then the \textbf{Clifford algebra}
$C(V,q)$ is defined as $C(V,q):=T(V)/\langle vw+wv-b(v,w) \rangle$ where $T(V)$ is the tensor algebra of $V$; cf.\  \cite[Section~4.3]{Kobayashi/Yoshino:2005}, \cite[Chapter 1, \S1]{Lawson/Michelsohn:1989}.

Let $V=\RR^{n}$ with standard basis vectors $v_i$, let $q=x_1^2+\cdots+x_{n}^2$.
Then in $C(V,q)$ we have $v_i^2=1$ and $v_iv_j=-v_jv_i$. 
 
\begin{proposition}[Cartan--Bott periodicity] \label{prop:cartan-bott}
For $n\geq 2$, the Clifford algebra $C(\RR^n,q)$ is isomorphic to the following algebra:
\begin{enumerate}
\setcounter{enumi}{-1}
\item $\RR\otimes_\RR \M{2^{\frac{n}{2}}}{\RR}$, if $n \equiv 0 \pmod 8$,
\item $\left( \RR \oplus \RR \right) \otimes_\RR \M{2^{\frac{n-1}{2}}}{\RR}$, if $n \equiv 1 \pmod 8$,
\item $\M{2}{\RR} \otimes_\RR \M{2^{\frac{n-2}{2}}}{\RR}$, if $n \equiv 2 \pmod 8$,
\item $\M{2}{\CC} \otimes_\RR \M{2^{\frac{n-3}{2}}}{\RR}$, if $n \equiv 3 \pmod 8$,
\item $\M{2}{\HH} \otimes_\RR \M{2^{\frac{n-4}{2}}}{\RR}$, if $n \equiv 4 \pmod 8$,
\item $\left( \M{2}{\HH} \oplus \M{2}{\HH} \right) \otimes_\RR \M{2^{\frac{n-5}{2}}}{\RR}$, if $n \equiv 5 \pmod 8$,
\item $\M{4}{\HH} \otimes_\RR \M{2^{\frac{n-6}{2}}}{\RR}$, if $n \equiv 6 \pmod 8$,
\item $\M{8}{\CC} \otimes_\RR \M{2^{\frac{n-7}{2}}}{\RR}$, if $n \equiv 7 \pmod 8$.
\end{enumerate}
\end{proposition}

\begin{proof}
See e.g.\ \cite[Proposition~4.4.1 + Table~4.4.1]{Kobayashi/Yoshino:2005}.
\end{proof}

Since $C(V,q)$ is an associative algebra, it becomes a Lie algebra by setting $[A,B]:=AB-BA$. 
With this in mind, Proposition~\ref{prop:cartan-bott} implies the following:

\begin{corollary} \label{cor:cartan-bott-max-cpt}
For $n\geq 2$, the maximal semisimple compact Lie subalgebra of the Clifford algebra $C(\RR^n,q)$ is isomorphic to the following Lie algebra:
\begin{enumerate}
\setcounter{enumi}{-1}
\item $\mathfrak{so}(2^{\frac{n}{2}})$, if $n \equiv 0 \pmod 8$,
\item $\mathfrak{so}(2^{\frac{n-1}{2}})\oplus\mathfrak{so}(2^{\frac{n-1}{2}})$, if $n \equiv 1 \pmod 8$,
\item $\mathfrak{so}(2^{\frac{n}{2}})$, if $n \equiv 2 \pmod 8$,
\item $\mathfrak{su}(2^{\frac{n-1}{2}})$, if $n \equiv 3 \pmod 8$,
\item $\mathfrak{sp}(2^{\frac{n-2}{2}})$, if $n \equiv 4 \pmod 8$,
\item $\mathfrak{sp}(2^{\frac{n-3}{2}})\oplus\mathfrak{sp}(2^{\frac{n-3}{2}})$, if $n \equiv 5 \pmod 8$,
\item $\mathfrak{sp}(2^{\frac{n-2}{2}})$, if $n \equiv 6 \pmod 8$,
\item $\mathfrak{su}(2^{\frac{n-1}{2}})$, if $n \equiv 7 \pmod 8$.
\end{enumerate}
\end{corollary}

\section{A lower bound on the dimension of a subalgebra}

\begin{definition}
For $n\geq 3$ let $\imp$ be the Lie subalgebra of $C(\RR^n,q)$ generated by $v_1v_2v_3$ and by $v_iv_{i+1}$, $1 \leq i < n$.
\end{definition}

\begin{lemma} \label{lowerbounddimension}
Let $n\geq 3$.
Then $\imp$ contains all products of the form $v_{j_1}v_{j_2}\cdots v_{j_k}$ for $2\leq k\leq n$ and $k \equiv 2, 3 \pmod 4$ with pairwise distinct $j_t \in \{ 1, \ldots, n \}$, with the possible exception of $v_1v_2\cdots v_n$. 
The exception can only happen if $n \equiv 3 \pmod 4$. 
\end{lemma}

\begin{proof}
It is well-known that all products $v_{j_1}v_{j_2}$, $j_1 \neq j_2$, are contained in $\imp$: Indeed, $\Lambda^2 \RR^n \cong \mathfrak{so}(n)$ (cf., e.g., \cite[Proposition~6.1]{Lawson/Michelsohn:1989}) is generated as a Lie algebra by the $v_iv_{i+1}$, $1 \leq i < n$ (cf., e.g., \cite[Theorem~1.31]{Berman}, \cite[Theorem~2.1]{Hainke/Koehl/Levy}).

Moreover, for pairwise distinct $j_t$, $1 \leq t \leq k+1$, one has 
\[[v_{j_1}v_{j_2},\;v_{j_2}v_{j_3}\cdots v_{k+1}] = 2v_{j_1}v_{j_3}\cdots v_{j_{k+1}}.\]
Since re-ordering of the factors simply yields scalar multiples, this shows inductively that, as long as $k+1 \leq n$, once an arbitrary factor of the form $v_{j_1}v_{j_2}\cdots v_{j_k}$ is contained in the Lie subalgebra, all factors of that form are contained in the Lie subalgebra. This statement is also true in the situation $k=n$, because in that case all factors of that form are scalar multiples of one another.   

We finally prove the claim by induction over $k$. For $k=2$ and $k=3$, this is obvious.
Suppose the claim holds for $k\equiv 3\pmod 4$, then the next value for $k$ to consider
is $k+3\equiv 2\pmod 4$. By induction hypothesis
$v_4v_5\cdots v_{k+3} \in \imp$ and
\[0 \neq [v_1v_2v_3,v_4v_5\cdots v_{k+3}] = 2 v_1v_2v_3v_4\cdots v_{k+3}. \]
If on the other hand the claim holds for $k\equiv 2\pmod 4$, then the next value for $k$ to consider
is $k+1\equiv 3\pmod 4$. If $k+2\leq n$, then by
induction hypothesis $v_3v_4\cdots v_{k+2}\in\imp$ and
\[0 \neq [v_1v_2v_3,v_3v_4\cdots v_{k+2}] = 2v_1v_2v_4\cdots v_{k+2}.\]
That is, the presence of all elements of the form $v_{j_1}v_{j_2}v_{j_3}$ with pairwise distinct $j_t \in \{ 1, \ldots, n \}$ inductively allows us to construct all elements of the form $v_{j_1}v_{j_2}\cdots v_{j_k}$ for $k \equiv 2, 3 \pmod 4$ with pairwise distinct $j_t \in \{ 1, \ldots, n \}$ for all $k \leq n$, with the possible exception of the situation $k=n \equiv 3 \pmod 4$, as the element $v_{k+2}$ does not exist in that case.
\end{proof}

\begin{remark}
It will turn out later, as a consequence of the proof of Theorem~\ref{thm:A} based on dimension arguments, that the above elements in fact generate $\imp$ as an $\RR$-vector space and that for $n \equiv 3 \pmod 4$ the element $v_1v_2\cdots v_n$ indeed is not contained in $\imp$, unless of course $n=3$.
\end{remark}

\begin{definition}
For $k\in\{0,1,2,3\}$, let
\[ \delta_k : \NN\to \NN : n \mapsto \sum_{\substack{i=0, \\ i \equiv k \pmod 4}}^n \binom{n}{i}. \]
\end{definition}

\begin{remark} \label{rem:delta_k}
Let $n\in\NN$ and let $M$ be a set of size $n$. Then the number of subsets of $M$ of size $k\pmod 4$ is precisely $\delta_k(n)$. Therefore
\[ \delta_0(n)+\delta_1(n)+\delta_2(n)+\delta_3(n)=2^n.\]
\end{remark}

\begin{consequence} \label{con:lowerbounddim-ineffective}
Let $n\geq 3$.
Then
\[
\dim \imp \geq
\begin{cases}
\delta_2(n) + \delta_3(n) & \text{ if }  n \not\equiv 3 \pmod 4 , \\
\delta_2(n) + \delta_3(n)-1 & \text{ if }  n \equiv 3 \pmod 4 .
\end{cases}
\]
\end{consequence}

\section{Combinatorics of binomial coefficients}

We now turn the lower bound from Consequence~\ref{con:lowerbounddim-ineffective} into a numerically explicit bound by deriving a closed formula in $n$ for the functions $\delta_k$.

\begin{proposition} \label{binomial}
Let $n \in \NN$ and $k\in\{0,1,2,3\}$.
\begin{enumerate}
\setcounter{enumi}{-1}
\item If $n \equiv 0 \pmod 4$, then
\[ \delta_k(n)
= \begin{cases}
  2^{n-2}  & \text{for } k \in \{ 1, 3 \}, \\
  2^{n-2}+(-1)^{\frac{n}{4}+ \frac{k}{2}} 2^{\frac{n}{2}-1}
           & \text{for }k \in \{ 0, 2 \}.
\end{cases}\]
\item If $n \equiv 1 \pmod 4$, then
\[ \delta_k(n)
= \begin{cases} 
  2^{n-2}+(-1)^{\frac{n-1}{4}}2^{\frac{n-3}{2}} & \text{for } k \in \{ 0, 1 \}, \\
  2^{n-2}-(-1)^{\frac{n-1}{4}}2^{\frac{n-3}{2}} & \text{for } k \in \{ 2, 3 \}.
  \end{cases}
\]
\item If $n \equiv 2 \pmod 4$, then
\[ \delta_k(n)
= \begin{cases} 
  2^{n-2} & \text{for } k \in \{ 0, 2 \}, \\
  2^{n-2}+(-1)^{\frac{n-2}{4}+ \frac{k-1}{2}} 2^{\frac{n}{2}-1}
          & \text{for } k \in \{ 1, 3 \}.
  \end{cases}
\]
\item If $n \equiv 3 \pmod 4$, then
\[ \delta_k(n)
= \begin{cases}         
  2^{n-2}-(-1)^{\frac{n-3}{4}}2^{\frac{n-3}{2}} & \text{for } k \in \{ 0, 3 \}, \\
  2^{n-2}+(-1)^{\frac{n-3}{4}}2^{\frac{n-3}{2}} & \text{for } k \in \{ 1, 2 \}.
  \end{cases}
\]
\end{enumerate}
\end{proposition}

\begin{proof}
Note first that the claimed identities hold for $n \in \{ 1, 2 \}$. The pairing $S \leftrightarrow S \triangle \{ m \}$, where $\triangle$ denotes symmetric difference, provides a bijection between the set of subsets of $M$ of even order with the set of subsets of $M$ of odd order. Combined with Remark~\ref{rem:delta_k} we conclude
\begin{equation}
\delta_0(n)+\delta_2(n) = \delta_1(n)+\delta_3(n) = 2^{n-1}.
  \label{eqn:halfsum}
\end{equation}

Moreover, the pairing $S \leftrightarrow M \setminus S$ provides a bijection
\begin{enumerate}[label={\rm(\roman*)}]
\item between the set of subsets of $M$ of order $1 \pmod 4$ and the set of subsets of $M$ of order $3 \pmod 4$, if $n \equiv 0 \pmod 4$,
\item between the set of subsets of $M$ of order $0 \pmod 4$ and the set of subsets of $M$ of order $1 \pmod 4$ and between the set of subsets of $M$ of order $2 \pmod 4$ and the set of subsets of $M$ of order $3 \pmod 4$, if $n \equiv 1 \pmod 4$,
\item between the set of subsets of $M$ of order $0 \pmod 4$ and the set of subsets of $M$ of order $2 \pmod 4$, if $n \equiv 2 \pmod 4$,
\item between the set of subsets of $M$ of order $0 \pmod 4$ and the set of subsets of $M$ of order $3 \pmod 4$ and between the set of subsets of $M$ of order $1 \pmod 4$ and the set of subsets of $M$ of order $2 \pmod 4$, if $n \equiv 3 \pmod 4$.
\end{enumerate}
Hence
\begin{align*}
\delta_1(n)&=\delta_3(n) &\text{ for } n \equiv 0 \pmod 4, \\
\delta_0(n)&=\delta_1(n) \quad\text{and}\quad
\delta_2(n) =\delta_3(n) &\text{ for } n \equiv 1 \pmod 4, \\
\delta_0(n)&=\delta_2(n) &\text{ for } n \equiv 2 \pmod 4, \\
\delta_0(n)&=\delta_3(n) \quad\text{and}\quad
\delta_1(n) =\delta_2(n) &\text{ for } n \equiv 3 \pmod 4.
\end{align*}

Together with Equation~\ref{eqn:halfsum}, this already yields the claim for (a), case $k \in \{ 1, 3 \}$ and for (c), case $k \in \{ 0, 2 \}$.

We will now prove case $k = 0$ of (b), (d) by induction, which by the above observations implies all claims made in (b), (d). Let $M$ be a set of order $n+2$ and let $a, b \in M$ be distinct elements so that $M = M' \cup \{ a, b \}$ for a set $M'$ of cardinality $n$. A subset $S \subset M$ of cardinality $0 \pmod 4$ satisfies exactly one of the following:
\begin{enumerate}[label={\rm(\roman*)}]
\item $S\subset M'$ has cardinality $0 \pmod 4$,
\item $S\setminus\{a\}\subset M'$ has cardinality $3 \pmod 4$,
\item $S\setminus\{b\}\subset M'$ has cardinality $3 \pmod 4$,
\item $S\setminus\{a,b\}\subset M'$ has cardinality $2 \pmod 4$.
\end{enumerate}
Hence for $n \equiv 1 \pmod 4$ resp.\ $n+2 \equiv 3 \pmod 4$ we have
\begin{align*}
\delta_0(n+2)
&= \delta_0(n) + \delta_2(n) + 2 \delta_3(n)
 = 2^{n-1} + 2 \delta_3(n) \\
&= 2^{n-1} + 2 \left(2^{n-2} - (-1)^{\frac{n-1}{4}}2^{\frac{n-3}{2}}\right) \\
&= 2^{n} - (-1)^{\frac{n-1}{4}}2^{\frac{n-1}{2}} \\
&= 2^{(n+2)-2} - (-1)^{\frac{(n+2)-3}{4}}2^{\frac{(n+2)-3}{2}},
\end{align*}
and similarly for $n \equiv 3 \pmod 4$ resp.\, $n+2 \equiv 1 \pmod 4$ we have
\begin{align*}
\delta_0(n+2)
&= \delta_0(n) + \delta_2(n) + 2 \delta_3(n)
 = 2^{n-1} + 2 \delta_3(n) \\
&= 2^{n-1} + 2\left(2^{n-2} - (-1)^{\frac{n-3}{4}}2^{\frac{n-3}{2}}\right) \\
&= 2^{n} - (-1)^{\frac{n-3}{4}}2^{\frac{n-1}{2}} \\
&= 2^{(n+2)-2} + (-1)^{\frac{(n+2)-1}{4}}2^{\frac{(n+2)-3}{2}}.
\end{align*}
Next we prove case $k=0$ of (a) using (c) as an induction hypothesis and afterwards case $k=1$ of (c) using (a) as an induction hypothesis. By the above observations this implies all claims made in (a) and (c).

In order to establish case $k=0$ of (a) we use the exact same combinatorial induction step as above and arrive again at
\begin{align*}
\delta_0(n+2)
&= \delta_0(n) + \delta_2(n) + 2 \delta_3(n)
 = 2^{n-1} + 2 \delta_3(n) \\
&= 2^{n-1} + 2\left(2^{n-2} + (-1)^{\frac{n-2}{4}+\frac{3-1}{2}}2^{\frac{n}{2}-1} \right) \\
&= 2^n + (-1)^{\frac{n+2}{4}}2^{\frac{n}{2}} \\
&= 2^{(n+2)-2} + (-1)^{\frac{n+2}{4}+\frac{0}{2}}2^{\frac{n+2}{2}-1}
\end{align*}
as claimed.

In order to establish case $k=1$ of (c) we use the same combinatorial induction step as above but need to observe that if $S\subset M$ is a subset of cardinality $1\pmod 4$, then $S \setminus \{ a, b \}$  may have cardinality $1 \pmod 4$, $3 \pmod 4$ or, in two different ways, $0 \pmod 4$. Therefore
\begin{align*}
\delta_1(n+2)
&= 2 \delta_0(n) + \delta_1(n) + \delta_3(n)
 = 2 \delta_0(n) + 2^{n-1} \\
&= 2\left(2^{n-2}+(-1)^{\frac{n}{4}+\frac{0}{2}}2^{\frac{n}{2}-1} \right) + 2^{n-1} \\
&= 2^n + (-1)^{\frac{n}{4}+\frac{0}{2}}2^{\frac{n}{2}} \\
&= 2^{(n+2)-2} + (-1)^{\frac{(n+2)-2}{4}+\frac{1-1}{2}}2^{\frac{n+2}{2}-1}.
\qedhere
\end{align*}
\end{proof}

Combining this with Consequence~\ref{con:lowerbounddim-ineffective} yields
the following:

\begin{consequence} \label{comparedimension}
Let $n \in \NN$ and $n\geq 2$.
\begin{enumerate}
\setcounter{enumi}{-1}
\item If $n \equiv 0 \pmod 8$, then
\begin{align*}
  \dim \imp
  \geq \delta_2(n) + \delta_3(n)
&= 2^{n-2} - 2^{\frac{n}{2}-1} + 2^{n-2}
= 2^{\frac{n-2}{2}}(2^{\frac{n}{2}}-1) \\
&= \dim_\RR(\mathfrak{so}(2^{\frac{n}{2}})).
\end{align*}
\item If $n \equiv 1 \pmod 8$, then
\begin{align*}
  \dim \imp
  \geq \delta_2(n) + \delta_3(n)
&= 2\left( 2^{n-2} - 2^{\frac{n-3}{2}} \right)
= 2^{\frac{n-1}{2}}(2^{\frac{n-1}{2}}-1) \\
&= \dim_\RR(\mathfrak{so}(2^{\frac{n-1}{2}}) \oplus \mathfrak{so}(2^{\frac{n-1}{2}})).
\end{align*}
\item If $n \equiv 2 \pmod 8$, then
\begin{align*}
  \dim \imp
  \geq \delta_2(n) + \delta_3(n)
&= 2^{n-2} + 2^{n-2} - 2^{\frac{n}{2}-1}
= 2^{\frac{n-2}{2}}(2^{\frac{n}{2}}-1) \\
&= \dim_\RR(\mathfrak{so}(2^{\frac{n}{2}})).
\end{align*}
\item If $n \equiv 3 \pmod 8$, then
\begin{align*}
  \dim \imp + 1
  \geq \delta_2(n) + \delta_3(n)
&= 2^{n-2} + 2^{\frac{n-3}{2}} + 2^{n-2} - 2^{\frac{n-3}{2}}
= 2^{n-1} \\
&= \dim_\RR(\mathfrak{su}(2^{\frac{n-1}{2}}))+1.
\end{align*}
\item If $n \equiv 4 \pmod 8$, then
\begin{align*}
  \dim \imp
  \geq \delta_2(n) + \delta_3(n)
&= 2^{n-2} + 2^{\frac{n}{2}-1} + 2^{n-2}
= 2^{\frac{n-2}{2}}(2^{\frac{n}{2}}+1) \\
&= \dim_\RR(\mathfrak{sp}(2^{\frac{n-2}{2}})).
\end{align*}
\item If $n \equiv 5 \pmod 8$, then
\begin{align*}
  \dim \imp
  \geq \delta_2(n) + \delta_3(n)
&= 2\left( 2^{n-2} + 2^{\frac{n-3}{2}} \right)
= 2^{\frac{n-1}{2}}(2^{\frac{n-1}{2}}+1) \\
&= \dim_\RR(\mathfrak{sp}(2^{\frac{n-3}{2}})\oplus \mathfrak{sp}(2^{\frac{n-3}{2}})).
\end{align*}
\item If $n \equiv 6 \pmod 8$, then
\begin{align*}
  \dim \imp
  \geq \delta_2(n) + \delta_3(n)
&= 2^{n-2} + 2^{n-2} + 2^{\frac{n}{2}-1}
= 2^{\frac{n-2}{2}}(2^{\frac{n}{2}}+1) \\
&= \dim_\RR(\mathfrak{sp}(2^{\frac{n-2}{2}})).
\end{align*}
\item If $n \equiv 7 \pmod 8$, then
\begin{align*}
  \dim \imp + 1
  \geq \delta_2(n) + \delta_3(n)
&= 2^{n-2} - 2^{\frac{n-3}{2}} + 2^{n-2} + 2^{\frac{n-3}{2}}
= 2^{n-1} \\
&= \dim_\RR(\mathfrak{su}(2^{\frac{n-1}{2}}))+1.
\end{align*}
\end{enumerate}
\end{consequence}

\section{Generalized spin representations of the split real $E_n$ series and the resulting quotients}

The example of a generalized spin representation of the maximal compact subalgebra of the split real Kac--Moody Lie algebra of type $E_{10}$ described in \cite{DamourKleinschmidtNicolai}, \cite{deBuylHenneauxPaulot}, \cite{Hainke/Koehl/Levy} generalizes directly to the whole $E_n$ series as follows.

Let $n \in \NN$, let $\mathfrak{g}$ be the split real Kac--Moody Lie algebra of type $E_n$, let $\mathfrak{k}$ be its maximal compact subalgebra, and let $X_i$, $1 \leq i \leq n$, be the Berman generators of $\mathfrak{k}$ (cf.\ \cite[Theorem~1.31]{Berman}, \cite[Theorem~2.1]{Hainke/Koehl/Levy}) enumerated in Bourbaki style as shown in  Figure~\ref{fig:En}, i.e., $X_1$, $X_3$, $X_4$, \ldots, $X_n$ belong to the $A_{n-1}$ subdiagram, generating $\mathfrak{so}(n)$, and $X_2$ to the additional node. As in Section~\ref{sec:cartan-bott} let $q$ be the standard positive definite quadratic form on $\RR^n$ and let $C = C(\RR^n,q)$ be the corresponding Clifford algebra, considered as a Lie algebra.

\begin{proposition}
Let $n\geq 3$. The assignment
\begin{itemize}
\item $X_{1} \mapsto v_1v_2$,
\item $X_2 \mapsto v_1v_2v_3$,
\item $X_j \mapsto v_{j-1}v_j$ for $3 \leq j \leq n$
\end{itemize}
defines a Lie algebra homomorphism $\rho$ from $\mathfrak{k}$ to the Lie subalgebra $\imp$ of $C$ generated by $v_1v_2v_3$ and by $v_iv_{i+1}$, $1 \leq i < n$, called the {\bf standard generalized spin representation of $\mathfrak{k}$}.
\end{proposition} 

\begin{proof}
The proof is based on the criterion established in \cite[Remark~4.5]{Hainke/Koehl/Levy} and is exactly the same as in the $E_{10}$ case discussed in \cite[Example~4.1]{Hainke/Koehl/Levy}.  
\end{proof}

\begin{proof}[Proof of Theorem~\ref{thm:A}]
By \cite[Theorem~4.11]{Hainke/Koehl/Levy} and since
$E_n$ is simply laced and connected for $n\geq 4$, the image $\imp$ of $\rho$ is semisimple and compact. By Lemma~\ref{lowerbounddimension} and Consequence~\ref{comparedimension}, $\dim_\RR(\imp)$ is at least as large as the dimension of the semisimple maximal compact Lie subalgebra of $C$ as given in Corollary~\ref{cor:cartan-bott-max-cpt}. The claim follows. 
\end{proof}

\begin{proof}[Proof of Theorem~\ref{thm:B}]
Let $\mathfrak{g}$ be a semisimple split real Lie algebra of type $E_4=A_4$, $E_5=D_5$, $E_6$, $E_7$ or $E_8$ and $\mathfrak{g} = \mathfrak{k} \oplus \mathfrak{a} \oplus \mathfrak{n}$ its Iwasawa decomposition. Since $\dim_\mathbb{R}(\mathfrak{k}) = \dim_\mathbb{R}(\mathfrak{n})$, from the combinatorics of the respective root system we conclude that the maximal compact Lie subalgebra $\mathfrak{k}$ has dimension
\begin{align*}
10 &= \frac{4\cdot 5}{2}
    = \frac{2^{\frac{4}{2}} \cdot (2^{\frac{4}{2}}+1)}{2}
    = \dim_\RR(\mathfrak{sp}(2))
    = \dim_\RR(\mathfrak{so}(5))
	&\text{ if } n=4, \\
20 &= 2 \cdot 10
    = \dim_\RR(\mathfrak{sp}(2)\oplus\mathfrak{sp}(2))
    = \dim_\RR(\mathfrak{so}(5)\oplus\mathfrak{so}(5))
	&\text{ if } n=5, \\
36 &= 4 \cdot 9
    = 2^{\frac{6-2}{2}}(2^{\frac{6}{2}}+1)
    = \dim_\RR(\mathfrak{sp}(2^{\frac{n-2}{2}})) 
	&\text{ if } n=6, \\
63 &= 2^6-1
    = \dim_\RR(\mathfrak{su}(8)) 
	&\text{ if } n=7, \\
120 &= \frac{16 \cdot 15}{2}
    = \frac{2^{\frac{8}{2}} \cdot (2^{\frac{8}{2}}-1)}{2}
    = \dim_\RR(\mathfrak{so}(16)) 
	&\text{ if } n=8.
\end{align*}
For $n\geq 4$ we may now apply Theorem~\ref{thm:A} and deduce that the
standard generalized spin representation $\rho$ has to be injective
in these cases.

This leaves the case $E_3=A_2\oplus A_1$. Since this diagram is not irreducible,
\cite[Theorem~4.11]{Hainke/Koehl/Levy}
only implies that $\im(\rho)=\imp$ is compact but not that it is
semisimple (and indeed, it is not). However, $n=3$ is also 
an exceptional case for Lemma~\ref{lowerbounddimension}. Taking that into consideration,
it follows that $\dim_\RR(\imp)\geq 2^2$ ($1, v_1v_2, v_2v_3, v_1v_2v_3$ is a basis of $\imp$). 
On the other hand, the Clifford algebra $C$ is isomorphic to $\M{2}{\CC}$,
hence $\kk\cong \mathfrak{u}(2)$, and this has dimension $4$.
Thus $\rho$ is also injective when $n=3$. The claim follows.
\end{proof}

\bibliographystyle{ralf}
\bibliography{References}

\end{document}